\documentclass{amsart}
\usepackage{tikz-cd}
\usepackage[unicode]{hyperref}
\usepackage{graphicx,color}
\usepackage{amssymb,amsmath,amsbsy}
\usepackage{mathrsfs}

\newtheorem{theorem}{Theorem}
\newtheorem{proposition}[theorem]{Proposition}
\newtheorem{corollary}[theorem]{Corollary}
\newtheorem{lemma}[theorem]{Lemma}
\theoremstyle{definition}

\newtheorem{example}[theorem]{Example}

\newcommand*\Bell{\ensuremath{\boldsymbol\ell}}
\newcommand{\Omegainj}{\Omega_{\mathrm{inj}}}
\newcommand{\operad}{\mathscr O}
\newcommand{\shiftoperator}{\mathfrak s}
\newcommand{\disk}{\mathbb D}
\newcommand{\initiallist}{\mathbf m}
\newcommand{\terminallist}{\mathbf M}
\newcommand{\incmap}{\mathfrak i}
\newcommand{\partialvz}{\partial_{v_0}}
\newcommand{\id}{\operatorname{id}}
\newcommand{\zero}{\mathbf 0}

\title{Configuration spaces form a Segal semi-dendroidal space}
\author{Philip Hackney}
\address{Department of Mathematics\\ Macquarie University\\ NSW 2109 \\ Australia}
\email{philip@phck.net} 

\begin{document}

\begin{abstract}
The purpose of this short note is to illustrate the utility of (semi-) dendroidal objects in describing certain `up-to-homotopy' operads. 
Specifically, we exhibit a semi-dendroidal space satisfying the Segal condition, whose evaluation at a $k$-corolla is the space of ordered configurations of $k$ points in the $n$-dimensional unit ball.
\end{abstract}


\maketitle

By forgetting radii, the $k$th space of the little $n$-disks operad is homotopy equivalent to the configuration space of $k$ ordered points in the unit $n$-disk. 
The latter collection of spaces (as $k$ varies) does not admit the structure of an operad, but in light of this homotopy equivalence should admit the structure of an up-to-homotopy operad.
Corollary \ref{main example} gives one way to make this precise using \emph{semi-dendroidal spaces} satisfying a \emph{Segal condition}.

The reader should recall the dendroidal category $\Omega$ from \cite{lectures,mw}; the former reference provides a good overview of basic dendroidal theory, and we will often adopt the same notation in this work.
We let $\Omegainj$ be the wide subcategory generated by isomorphisms and face maps.
A space-valued presheaf on $\Omegainj$ (resp.\ $\Omega$) will be called a semi-dendroidal (resp.\ dendroidal) space.
A planar structure on a (rooted) tree is an assignment, for each tree $T$ and each vertex $v\in T$, a bijection $b_v : \{ 1, \dots, k_v \} \to in(v)$. 
For notational reasons it will be convenient to assume that any given tree comes equipped with a planar structure, though we will not require maps to preserve this extra structure (so we are actually working with the equivalent category which was called $\Omega'$ in Example 2.8 of \cite{bm}).
Given any colored (symmetric) operad $\operad$ in topological spaces, there is an associated dendroidal space (the \emph{dendroidal nerve}) $N_d(\operad)$ with $N_d(\operad)_T = Oper(\Omega(T), \operad)$. 
A point in $N_d(\operad)_T$ may be identified with a pair $(f_0, f_1)$, where $f_0 : E(T) \to col(\operad)$ is a function and $f_1$ assigns to each vertex $v$ of $T$ a point of $\operad(f_0b_v(1), \dots, f_0b_v(k_v); f_0(out(v)))$.

Let $\operad$ be a 2-colored operad in the category of topological spaces; we write $\{ 1,2\}$ as the color set of $\operad$. We ask that $\operad$ satisfies
\begin{equation}\label{one-focused}
	\operad(\Bell; 2) = \begin{cases}
		\ast & \text{if } \Bell = 2 \\
		\varnothing & \text{otherwise}.
	\end{cases}
\end{equation}
Here we are using the notation $\Bell = l_1l_2\ldots l_p$ with $l_i \in \{ 1,2\}$ and $|\Bell| = p\geq 0$ for (ordered) lists in the set $\{ 1,2\}$.
Given two such lists $\Bell = l_1l_2\ldots l_p$ and $\Bell' = l_1'l_2'\ldots l_q'$, write $\Bell \circ_i \Bell' = l_1 \ldots l_{i-1} l_1' \ldots l_q' l_{i+1} \ldots l_p$.
There is a natural partial order on the set of lists of a fixed length, given by entrywise comparison: $l_1\dots l_p \leq l_1' \ldots l_p'$ if and only if $l_i \leq l_i'$ for all $i$.
We will write $\initiallist^p$ (resp.\ $\terminallist^p$) for the list of length $p$ with every entry $1$ (resp.\ $2$).
The operad $\operad$ should come equipped with a collection of weak homotopy equivalences $\shiftoperator_{\Bell,\Bell'} : \operad(\Bell; 1) \to \operad(\Bell'; 1)$ whenever $\Bell \leq \Bell'$. These maps should respect the partial order, that is
\begin{equation*}
	\shiftoperator_{\Bell', \Bell''} \circ \shiftoperator_{\Bell, \Bell'} = \shiftoperator_{\Bell, \Bell''} \qquad \& \qquad \shiftoperator_{\Bell, \Bell} = \id_{\operad(\Bell; 1)}.
\end{equation*}
Further, these should be compatible with the operad structure, in the sense that  
\begin{equation}\label{operad compatibility}
	(\shiftoperator_{\Bell, \Bell'} x) \circ_i (\shiftoperator_{\Bell'', \Bell'''} y) = \shiftoperator_{\Bell \circ_i \Bell'', \Bell' \circ_i \Bell'''} (x \circ_i y)
\end{equation}
whenever both sides are defined (that is, whenever $\Bell \leq \Bell'$, $\Bell'' \leq \Bell'''$, and $l_i = l_i' = 1$) and
$\sigma^* \circ\shiftoperator_{\Bell, \Bell'} = \shiftoperator_{\Bell \cdot \sigma, \Bell' \cdot \sigma}\circ \sigma^*$\label{permutation thing} for $\sigma \in \Sigma_p$.
Such a 2-colored operad $\operad$ along with the data $\{\shiftoperator_{\Bell, \Bell'}\}$ shall be called an \emph{operad with shifts}. 

\begin{example}
Our main example is the \emph{operad of points and little disks}. 
The spaces $\operad(\Bell; 2)$ are determined by \eqref{one-focused}, while a point of $\operad(\Bell; 1)$ consists of $|\Bell|$ pieces of data:
\begin{itemize}
	\item If $l_i = 1$, an affine embedding of the open unit disk $\disk \subset \mathbb R^n$ into $\disk$; we will write $a_i : \disk \to \disk$ for the embedding and $D_i$ for its image.
	\item If $l_i = 2$, a point $d_i \in \disk$.
\end{itemize}
Writing $\overline{D}_i$ for the closure of $D_i$, 
these data should be a configuration in the sense that, unless $i=j$, we have $D_i \cap D_j = \varnothing$, $d_i \neq d_j$, and $d_i \notin \overline{D}_j$. In order to topologize $\operad(\Bell; 1)$, notice that each affine map $a_i(t) = r_it+c_i$ may be identified with a point $(r_i, c_i) \in \mathbb R_{>0} \times \mathbb R^n$, while $d_j \in \disk \subset \mathbb R^n$. 
We thus regard $\operad(\Bell; 1)$ as a subspace of $\mathbb R^{p_1(n+1) + p_2n}$, where $p_j = |\{ i \,|\, l_i = j\}|$. 
The operad structure is a variation on the usual one for the little $n$-disks operad: when $l_i = 1$, the map
\[
	\circ_i : \operad(\Bell; 1) \times \operad(\Bell'; 1) \to \operad(\Bell \circ_i \Bell'; 1) 
\]
is given on a pair $(\mathbf x, \mathbf y)$ by applying the affine transformation $x_i = a_i : \disk \to \disk$ to all of the disks and points that constitute $\mathbf y$, to end up with \[ \mathbf x \circ_i \mathbf y = x_1, \dots, x_{i-1}, a_i y_1, \dots, a_i y_{|\Bell'|}, x_{i+1}, \dots, x_{|\Bell|}. \]
This is an operad with shifts: define $\shiftoperator_{\Bell,\Bell'}(\mathbf x) = \mathbf y$ componentwise as follows. If $l_i = l_i'$ then set $y_i = x_i$. If $l_i < l_i'$, then $l_i = 1$ and $l_i'=2$, so $x_i$ is a disk embedding $a_i$, while $y_i$ is meant to to be a point. In this case, we set $y_i = a_i(0)$, the center of the disk $x_i$. 
At the two extremes, we have that $\operad( \initiallist^k; 1) = \operad( 1 \overset{k}\cdots 1; 1)$ is the $k$-th space of the usual little $n$-disks operad and $\operad( \terminallist^k ; 1) = \operad( 2 \overset{k}\cdots 2; 1)$ is the configuration space of $k$ points in $\disk$.

To show that $\shiftoperator_{\Bell, \Bell'}$ is a homotopy equivalence, it is enough do so when $\Bell < \Bell'$ with $l_i = l_i'$ for all $i$ except a single $i_0$ where $1 = l_{i_0} < l_{i_0}' = 2$. Define a continuous $\epsilon : \operad(\Bell'; 1) \to \mathbb R_{>0}$ by letting $\epsilon(\mathbf y)$ be the minimum of the distances of $d_{i_0} = y_{i_0}$ to $d_i$ ($i\neq i_0$), $D_i$, and $\partial \disk$. Let $g : \operad(\Bell'; 1) \to \operad(\Bell; 1)$, $g(\mathbf y) = \mathbf x$ be the right inverse to $\shiftoperator_{\Bell, \Bell'}$ which is given by $x_i = y_i$ for $i\neq i_0$, and $x_{i_0}$ is the affine embedding $u \mapsto \frac{1}{2} \epsilon(\mathbf y) u + d_{i_0}$.
Writing $x_{i_0} = (u \mapsto r_{\mathbf x} u + c_{\mathbf x})$, define
\[
	H_t(\mathbf x) = x_1, \dots, x_{i_0-1}, \left( u \mapsto \left[ \left(tr_{\mathbf x} + (1-t) \frac{1}{2}\epsilon(\shiftoperator_{\Bell, \Bell'}(\mathbf x))\right)u + c_{\mathbf x}\right] \right), x_{i_0+1}, \dots, x_{|\Bell|}
\]
which is a homotopy from $g \circ \shiftoperator_{\Bell, \Bell'}$ to the identity of $\operad(\Bell; 1)$.

\end{example}

Given any operad with shifts $(\operad, \shiftoperator_{\Bell, \Bell'})$, we will presently define an associated semi-dendroidal space $X$.
Let $X_\eta = *$.
The most concise description of the spaces $X_T$, $T\neq \eta$, are as subspaces of $N_d(\operad)_T = Oper(\Omega(T), \operad)$.
We say that a colored operad map $\Omega(T) \to \operad$ is in $X_T$ if and only if the map on color sets $E(T) \to \{1,2\}$ sends all of the leaves to 2, the root to 1, and all of the internal edges to 1. 
Define two subspaces $X_T^L$ and $X_T^{IR}$ of $N_d(\operad)_T$. The space $X_T^L$ is the subspace consisting of those maps which send all leaves to 2, and $X_T^{IR}$ is the subspace consisting of those maps which send the root and all internal edges to 1. If $T\neq \eta$, then $X_T = X_T^L \cap X_T^{IR}$.
For use in later equations, we will write $\incmap' = \incmap'_T : X_T^L \to N_d(\operad)_T$, $\incmap'' = \incmap''_T : X_T \to X_T^L$, and $\incmap = \incmap' \circ \incmap'' = \incmap_T : X_T \to N_d(\operad)_T$ for the inclusions. 

For each non-trivial tree $T$, there is a map $\shiftoperator_T : N_d(\operad)_T \to X_T^L$.
Fix a planar structure on $T$, which allows us to identify $N_d(\operad)_T$ with the space of pairs $(f_0,f_1)$ as in the first paragraph.
Writing $\shiftoperator_T(f_0,f_1) = (f_0',f_1')$, we first set $f_0'(e) = 2$ if $e$ is a leaf and $f_0'(e) = f_0(e)$ otherwise.
If $v$ is a vertex, write $\Bell_v = f_0 b_v 1, \dots f_0 b_v k_v$ and $\Bell'_v = f_0' b_v 1, \dots, f_0' b_v k_v$. We have $\Bell_v \leq \Bell_v'$ for all $v$, so we can define $f_1'(v) = \shiftoperator_{\Bell_v, \Bell_v'}f_1(v)$.

Let $\alpha : S \to T$ be a map of $\Omega$ satisfying the following property: if $v$ is a vertex of $S$ so that $\alpha(v)$ is an edge $e$ of $T$, then $e$ is not a leaf edge.
Note that this property is not closed under composition, but every map in $\Omegainj$ satisfies it.
If $\alpha $ is such a map and $S\neq \eta$, then the composite $X_T^{IR} \hookrightarrow N_d(\operad)_T \xrightarrow{\alpha^*} N_d(\operad)_S \xrightarrow{\shiftoperator_S} X_S^{L}$ actually lands in the subspace $X_S \subseteq X_S^L$.
To distinguish from the operator $\alpha^*$ in the dendroidal nerve, we will write $\hat \alpha : X_T \to X_S$ for the 
map defined by $\incmap''_S \hat \alpha = \shiftoperator_S \alpha^* \incmap_T$ (equivalently $\incmap \hat \alpha = \incmap' \shiftoperator \alpha^* \incmap$).
If $S=\eta$, then $\alpha : \eta \to T$ automatically satisfies the indicated property; we will write $\hat \alpha$ for the unique function $X_T \to X_\eta = *$.

We wish to show that the relations which hold among maps in $\Omegainj$ also hold among the hat-maps. 
For this purpose it would be enough to show that hat anticommutes with composition, which we can show in several cases.

\begin{lemma}\label{the lemma}
	Let $\alpha : S \to T$ and $\beta : R \to S$ be two maps of $\Omegainj$ so that $\alpha$ induces a bijection on leaves. Then $\hat \beta \hat \alpha = \widehat{\alpha \circ \beta}$. The same equality holds if $\alpha$ is an arbitrary map of $\Omegainj$ and $\beta$ is an isomorphism.
\end{lemma}
\begin{proof}
	Since $\alpha$ sends leaves to leaves, $\alpha^* \incmap: X_T \to N_d(\operad)_T \to N_d(\operad)_S$ already lands in the subspace $X_S^L$, and on this subspace $\incmap' \shiftoperator (x) = x$. Thus $\alpha^*\incmap = \incmap' \shiftoperator \alpha^* \incmap = \incmap \hat \alpha$.
	This equality implies the second in $\incmap \hat \beta \hat \alpha = \incmap' \shiftoperator \beta^* \incmap \hat \alpha = \incmap' \shiftoperator \beta^* \alpha^* \incmap = \incmap' \shiftoperator (\alpha \circ \beta)^* \incmap = \incmap \widehat{\alpha \circ \beta}$. Since $\incmap$ is an injection, the conclusion follows.

	Let us address the second statement.
	Write $\beta^*_L : X^L_S \to X^L_R$ for the restriction of $\beta_*$.
	It is immediate that $\incmap'_R \beta^*_L = \beta^* \incmap'_S$.
	Since $\sigma^* \circ\shiftoperator_{\Bell, \Bell'} = \shiftoperator_{\Bell \cdot \sigma, \Bell' \cdot \sigma}\circ \sigma^*$ whenever $\sigma \in \Sigma_p$ (see page \pageref{permutation thing}), we have $\shiftoperator_S \beta^* = \beta^*_L \shiftoperator_R$.
	As in the first paragraph, $\beta^*\incmap  = \incmap \hat \beta$ because $\beta$ sends leaves to leaves.
    Putting these three facts together, we have $
	\incmap \widehat{\alpha \circ \beta} = 
	\incmap' \shiftoperator (\alpha \circ \beta)^* \incmap = 
	\incmap' \shiftoperator  \beta^* \alpha^* \incmap =
	\incmap' \beta^*_L \shiftoperator \alpha^* \incmap =
	\beta^* \incmap' \shiftoperator \alpha^* \incmap =
	\beta^* \incmap \hat \alpha = \incmap \hat \beta \hat \alpha$, hence $\hat \beta \hat \alpha = \widehat{\alpha \circ \beta}$.
\end{proof}

\begin{theorem}
The collection $\{ X_T \}$ together with the operators $\hat \alpha$ for $\alpha \in \Omegainj$ constitute a semi-dendroidal space.
\end{theorem}
\begin{proof}
It is enough to show, given a commutative diagram 
\begin{tikzcd}
	T_0 \rar{\delta} \dar{\partial'} & T_1 \dar{\partial} \\
	T_2 \rar{\delta'} & T_3
\end{tikzcd}
where $\partial', \partial$ are face maps 
and $\delta, \delta'$ are either both face maps or isomorphisms, that $\hat \delta \hat \partial = \hat \partial' \hat \delta'$.
This certainly follows if $\hat \delta \hat \partial = \widehat{\partial \circ \delta}$ and $\hat \partial' \hat \delta' = \widehat{\delta' \circ \partial'}$.
Several cases of these equalities have been established in lemma \ref{the lemma}; note also that they are both obvious if $T_0 = \eta$ since $X_\eta = *$. We will now sweep up the few remaining cases.

Consider a composition $T_0 \overset{\delta}\to T_1 \overset{\partialvz}\to T_2$ of face maps with $T_0 \neq \eta$ and $\partialvz$ an outer face map which chops off a vertex $v_0$ whose incoming edges are leaves.
Since $|T_2| > 1$, the output of $v_0$ is not the root edge, so $e_0 = out(v_0)$ is the $i_0$th input of some other vertex $w_0$.

Fix an arbitrary $(f_0,f_1) \in X_{T_2} \subseteq N_d(\operad)_{T_2}$
and write $\shiftoperator_{T_1}\partialvz^*(f_0,f_1) = (g_0, g_1)$, $\shiftoperator_{T_0}\delta^* (g_0,g_1) = (h_0, h_1)$, and $\shiftoperator_{T_0}(\partialvz \circ \delta)^*(f_0,f_1)  = (h_0', h_1')$.
Since $(h_0,h_1)$ and $(h_0',h_1')$ are both in $X_{T_0}$, we know that $h_0 = h_0'$ (both send all leaves to $2$ and all other edges to $1$).
Showing $(h_0, h_1) = (h_0', h_1')$ is the same as showing $h_1 = h_1'$, and we split this task into several cases.
Notice that since $\partialvz$ is external, we can immediately compute $g_1$:  
\begin{equation}\label{equation g}
	g_1 (v) = \begin{cases}
		f_1(v) & v\neq w_0 \\
		\shiftoperator_{\Bell, \Bell'} f_1(w_0) & v = w_0
	\end{cases}
\end{equation}
where $\Bell$ and $\Bell'$ are identical except at entry $i_0$. 
It will be convenient to write $\shiftoperator_{\Bell, \Bell'}$ as $\shiftoperator^{i_0}$, indicating which entry has changed. More generally, given $\Bell \leq \Bell'$, let $I$ be the set so that $i\in I$ if and only if $l_i < l_i'$ and write $\shiftoperator_{\Bell, \Bell'} = \shiftoperator^I$. 
We can then rewrite \eqref{equation g} using the planar structure as $g_1(v) = \shiftoperator^{b^{-1}_{v}(e_0)}f_1(v)$.

In all three cases below, write $(\partialvz \circ \delta)^*(f_0,f_1) = (\tilde h_0, \tilde h_1)$.

\noindent \textbf{Case 1: $\delta$ is external at the root.}
As in the proof of lemma \ref{the lemma}, $(h_0,h_1) = \delta^*(g_0,g_1)$ since $\delta$ is a bijection on leaves. Further, since $\delta$ is external, $h_1$ is a restriction of $g_1$. Since $\partialvz$ and $\delta$ are external, we have $\tilde h_1 (v) = f_1(v)$ whenever the left hand side is defined.
Calculating $h_1'(v)$, we see that it is just $f_1(v)$ unless $v = w_0$, in which case we have $h_1'(w_0) = \shiftoperator_{\Bell, \Bell'} f_1(w_0).$ Thus $h_1'(v) = g_1(v) = h_1(v)$.

\noindent \textbf{Case 2: $\delta$ is external at a leaf vertex $v_1$.}
Let $e_1$ be the output edge of $v_1$.
As above, $\tilde h_1$ is just a restriction of $f_1$. We have $h_1'(v) = \shiftoperator^{b_v^{-1}\{e_0, e_1\}} f_1(v) = \shiftoperator^{b_v^{-1}(e_1)} \shiftoperator^{b_v^{-1}(e_0)} f_1(v) = \shiftoperator^{b_v^{-1}(e_1)} g_1(v) = h_1(v)$, so $h_1 = h_1'$.

\noindent \textbf{Case 3: $\delta$ is internal at an edge $e_1$.}
As in the proof of lemma \ref{the lemma}, $(h_0,h_1) = \delta^*(g_0,g_1)$ since $\delta$ is a bijection on leaves.
We will write $v_1$ and $w_1$ for the two vertices that $e_1$ connects, with $e_1 = out(v_1)$ and $h_{w_1}(i_1) = e_1 \in in(w_1)$. It is possible that $w_0 = v_1$ or $w_0 = w_1$. Write $V(T_0) =\{ \bar v \} \sqcup V(T_1) \setminus \{v_1, w_1\}$ with $\delta(\bar v) = w_1 \circ_{e_1} v_1$.
We have that $\tilde h_1(\bar v) = f_1(w_1) \circ_{e_1} f_1(v_1)$ and otherwise $\tilde h_1(v) = f_1(v)$. 
Then
\[
	h_1'(\bar v) = \shiftoperator^{b^{-1}_{\bar v}(e_0)}(f_1(w_1) \circ_{i_1} f_1(v_1)) \overset{\eqref{operad compatibility}}= [\shiftoperator^{b^{-1}_{w_1}(e_0)} f_1(w_1)] \circ_{i_1} [\shiftoperator^{b^{-1}_{v_1}(e_0)} f_1(v_1)] = g_1(w_1) \circ_{i_1} g_1(v_1) = h_1(\bar v)
\]
and otherwise  $h_1'(v) = \shiftoperator^{b^{-1}_{\bar v}(e_0)} f_1(v) = g_1(v) = h_1(v)$.
Thus $h_1 = h_1'$.
\end{proof}

Given any (semi-)dendroidal space $Z$, there is the \emph{Segal map},
\[
	Z_T \to \prod_{v\in T} Z_{C_v}
\]
induced by the corolla inclusions $C_v \to T$ as $v$ ranges over all vertices of $T$.
If $Z$ is the dendroidal nerve of a one-colored operad, then the Segal map is an isomorphism.
It is interesting to weaken this and ask that the Segal map is merely a weak equivalence -- as in \cite[Theorem 1.1]{bh1} and \cite[Section 9]{cm-simpop} we expect this notion to be closely related to one-colored topological operads.

\begin{theorem}
Suppose that $\operad$ is an operad with shifts and $X$ is the associated semi-dendroidal space.
Then $X$ satisfies the \emph{Segal condition}, that is, for $T \neq \eta$, the Segal map $X_T \to \prod_{v\in T} X_{C_v}$ is a weak equivalence and $X_\eta = *$.
\end{theorem}
\begin{proof}
By definition, $X_\eta$ is a point.
Let $T$ be a nontrivial tree, and let $f: E(T) \to \{ 1,2\}$ be the function which takes the leaves to $2$, the internal edges to $1$, and the root to $1$.
Write $\Bell_v = fb_v(1), \dots, fb_v(k_v)$. 
Then $X_T = \prod_{v \in T} \operad(\Bell_v; 1)$.
Suppose that $w \in T$ and let $\alpha: C_{w} \to T$ be the corolla inclusion. Then $\hat \alpha : X_T \to X_{C_{w}}$ is the composite
\[
	X_T = \prod_{v \in T} \operad(\ell_v; 1) \xrightarrow{\pi_w} \operad(\Bell_w; 1) \xrightarrow{\shiftoperator_{\ell_w,\terminallist^{|w|}}} \operad(\terminallist^{|w|}; 1) = X_{C_w}.
\]
Then the Segal map $X_T = \prod_{v} \operad(\Bell_v; 1) \to \prod_v \operad(\terminallist^{|v|}; 1) = \prod_v X_{C_v}$
is the product of weak homotopy equivalences, hence a weak homotopy equivalence.
\end{proof}


\begin{corollary}\label{main example}
	There is a semi-dendroidal space $X$ satisfying the Segal condition so that $X_{C_k}$ is the configuration space of $k$ points in $\disk$.
\end{corollary}
\begin{proof}
Apply the previous theorem to the operad with shifts given by configurations of points and disks in the unit disk. As we mentioned above, $X_{C_k} = \operad(\terminallist^k; 1)$ is the ordered configuration space of $k$ points in the disk.  
\end{proof}

It is natural to ask whether the semi-dendroidal space $X$ admits the structure of a dendroidal space, that is, whether one can define degeneracy operators which are compatible with the existing face maps.
The reader may have noticed that we have already defined $\hat \alpha : X_T \to X_S$ for many maps of $\Omega$ which were not in $\Omegainj$, including all degeneracy maps except those that are degenerate at a leaf. Further, the proof of lemma \ref{the lemma} shows that many of the expected relations among faces and degeneracies hold with these definitions.

Nevertheless, we will now show that we cannot in general extend to a dendroidal structure.

\begin{proposition}\label{no dendroidal}
	The semi-dendroidal space $X$ from corollary \ref{main example} does not admit the structure of a dendroidal space.
\end{proposition}

We devote the remainder of the paper to the proof of this proposition. 
We will prove this by looking at its underlying semi-simplicial space (also called $X$) and showing that no choice of degeneracy operators gives a simplicial space. 
For convenience, we will write points of $X_k$ as lists $[a_1, \dots, a_{k-1}, P]$, where $a_1, \dots, a_{k-1} \in \operad(1;1) \subset (0,1] \times \disk$ and $P \in \operad(2;1) = \disk$.
Each $a_i$ is an embedding $\disk \to \disk$ of the form $a_i(x) = r_ix + c_i$ where $c_i \in \disk \subset \mathbb R^n$ and $r_i > 0$. 
For $k\geq 1$, the face maps are
\[
	d_i [a_1, \dots, a_{k-1}, P] = \begin{cases}
		[a_2, \dots, a_{k-1}, P] & i = 0 \\
		[a_1, \dots, a_i \circ a_{i+1}, \dots, a_{k-1}, P] & 1\leq i \leq k-2 \\
		[a_1, \dots, a_{k-2}, a_{k-1}(P)] & i = k-1 \neq 0 \\
		[a_1, \dots, a_{k-2}, a_{k-1}(\zero)]  & i = k.
	\end{cases}
\]

We now attempt to construct degeneracy operators in low degrees, and eventually show they cannot be chosen to satisfy all of the simplicial identities. We only need information about three of the degeneracy maps, namely $s_i : X_i \to X_{i+1}$ for $i=0,1,2$.
The map $s_0: * = X_0 \to X_1 = \disk$ just picks out a point, which we will call $A$ for the moment.
Let us examine $s_1 : X_1 \to X_2$.
Since $d_0 s_1 [P] = s_0d_0[P] = [A]$, we have $s_1[P] = [r^P x + c^P, A]$, where $r: \disk \to (0,1]$ and $c: \disk \to \disk$ are continuous maps satisfying $0 < r^P \leq dist(c^P, \partial \disk)$ for all $P$.
Since $[P] = d_2 s_1[P]$, we have $P = r^P\zero + c^P = c^P$, hence $c$ is the identity map on $\disk$ and $s_1[P] = [r^Px + P, A]$. Finally, since $[P] = d_1 s_1[P] = [r^PA + P]$, we conclude that $A = \zero$.
Thus $s_0[\,] = [\zero]$ and $s_1[P] = [r^Px + P, \zero]$ for some unspecified function $r^P$.

We now turn to $s_2 : X_2 \to X_3$. We immediately know that $s_2[a,P]$ is of the form
\[
	s_2[a,P] = [a, R^{a,P}x + C^{a,P}, Q^{a,P}]
\]
by examining the first term of $[a,P] = d_3s_2[a,P]$. Here, $R : X_2 \to (0,1]$, $C: X_2 \to \disk$, and $Q: X_2 \to \disk$ are continuous functions.
The second term of $[a, P] = d_3s_2[a,P]$ is $R^{a,P} \zero + C^{a,P}$, so $C^{a,P} = P$.
Comparing the second entries in  
\[
	[a, P] = d_2 s_2 [a,P] = d_2[a, R^{a,P}x + P, Q^{a,P}] = [a, R^{a,P}Q^{a,P} + P],
\]
we see that $Q^{a,P} = \zero$ since $R^{a,P}$ is never zero.
Thus $s_2[a, P] = [a, R^{a,P}x + P, \zero]$ for some function $R^{a,P}$.

In giving form to $s_2$, we used only the identities $d_3 s_2 = \id = d_2 s_2$. 
The main trouble is with the identity $d_1 s_2 = s_1 d_1$.
We can calculate (writing $a(x) = r_1 x + c_1$)
\[
	d_1s_2[a,P] = d_1 [a, R^{a,P}x + P, \zero] = [r_1 R^{a,P} x + r_1 P + c_1, \zero]
\]
and
\[
	s_1 d_1 [a,P] = s_1 [a(P)] = [r^{a(P)} x + a(P), \zero].
\]
If $d_1s_2 = s_1d_1$, then we would have $r_1 R^{a,P} = r^{a(P)}$ for all $a$ and $P$. 
In the next paragraph, we will show that this is not possible.

Fix $P, c_1\in \disk$ and let $B = 1 - |c_1|$. Write $f(t) = R^{tx + c_1, P}$ and $g(t) = r^{tP+c_1}$, which are positive real-valued functions.
The function $f$ is defined on $(0, B]$ and is bounded $|f(t)| \leq 1$.
The function $g$ is defined on the closed interval $[0, B]$ and is continuous from the right at zero.
Assuming $d_1s_2 = s_1 d_1$ in the previous paragraph implies that $tf(t) = g(t)$ for all $t\in (0, B]$. Since $f$ is bounded, the left hand side approaches $0$ as $t$ goes to zero, while the limit of the right hand side is $g(0) = r^{c_1} > 0$. This is a contradiction, hence $d_1s_2 \neq s_1 d_1$.

\subsection*{Acknowledgments} 
I'd like to thank Jonas Hartwig for encouraging me to write this paper in the first place, John Bourke and Steve Lack for being interested in this work over the past few months, and Christopher Lustri for answering a question at a key moment.
I would especially like to thank Gabriel C.\ Drummond-Cole, as we worked through something close to the proof of proposition \ref{no dendroidal} together.

\bibliographystyle{amsplain}
\bibliography{css}

\end{document}